\documentclass{article}
\usepackage{amssymb,amsmath,bm,paralist,graphics,epsfig,graphicx,epstopdf,amsthm,geometry,color,yhmath}
\usepackage{cite}
\usepackage[colorlinks,urlcolor=blue]{hyperref}

\theoremstyle{plain}
\newtheorem{theorem}{Theorem}[section]

\newtheorem{remark}[theorem]{Remark}

\newcommand{\norm}[1]{\left\Vert#1\right\Vert}
\newcommand{\brac}[1]{\left(#1\right)}
\newcommand{\abs}[1]{\left\vert#1\right\vert}

\newcommand{\ie}{{\it{i.e.}}}
\newcommand{\eg}{{\it{e.g.\ }}}

\newcommand{\diag}{\mbox{diag}}

\usepackage{lineno}

\usepackage{enumitem}

\begin{document}

\baselineskip=1pc  

\begin{center}
{\bf \large Validity of relaxation models arising from numerical schemes for hyperbolic-parabolic systems}
\end{center}

\vspace{.2in}

\centerline{
Zhiting Ma \footnote{Beijing Institute of Mathematical Sciences and Applications, Beijing 101408, China. E-mail: mazt@bimsa.cn.}
\qquad
Weifeng Zhao
\footnote[2]{Corresponding author. Department of Applied Mathematics, University of Science and Technology Beijing, Beijing 100083, China. wfzhao@ustb.edu.cn}
}

\vspace{.4in}

\centerline{\bf Abstract}

\vspace{.2in}

This work is concerned with relaxation models arising from numerical schemes for hyperbolic-parabolic systems. Such models are a hyperbolic system with both the hyperbolic part and the stiff source term involving a small positive parameter, and thus are endowed with complicated multiscale properties.
Relaxation models are the basis of constructing corresponding numerical schemes and a critical issue is the convergence of their solutions to those of the given target systems, the justification of which is still lacking.  In this work, we formulate convergence criteria for general hyperbolic relaxation systems to validate relaxation models in numerical schemes of hyperbolic-parabolic systems. By verifying the convergence criteria, we demonstrate the convergence, and thereby the approximation validity, of five representative relaxation models, providing a solid basis for the effectiveness of the corresponding numerical schemes. 
Moreover, we propose a new relaxation model for the general multi-dimensional hyperbolic-parabolic system. With some mild assumptions on the system, we show that the proposed model satisfies the convergence criteria. We remark that the existing relaxation models are constructed only for a special case of hyperbolic-parabolic system, while our new relaxation model is valid for general systems.

\vfill


\vspace{.2in}


\vspace{.2in}

\section{Introduction}

Over the past decades, relaxation models have emerged as a powerful tool for mathematical modeling and numerical simulations \cite{Jin2022}.
These models are usually proposed to approximate a given system of partial differential equations (PDEs) with a larger, but often simpler and more tractable system by introducing auxiliary variables and stiff source terms with small relaxation parameters. 
Typical examples include the Jin-Xin relaxation system for hyperbolic conservation laws \cite{jin1995relaxation,jin1998Diffusion,jin1998jnadiffusive} and kinetic approximations for fluid dynamics equations \cite{arun2007genuinely,arun2013,guo2013lattice}. 
A crucial issue for relaxation models is the convergence (or compatibility) to the given original system in the relaxation limit, which ensures that the relaxation approximation faithfully captures the dynamics of the underlying system. 
On this issue, fundamental results have been established 
with the aid of energy estimates, entropy structures, and compensated compactness techniques; see \eg \cite{CLL1994,yong1999singular,jin1998Diffusion,bouchut2000diffusive,lattanzio2001hyperbolic,LN2002,peng2016parabolic,peng2025convergence}.

In this work, we are concerned with relaxation models arising from numerical schemes for hyperbolic-parabolic systems of the form
\begin{equation}\label{equ:hp-equ}
\begin{aligned}
    \partial_t u + \sum^d_{j=1} a_j(u) \partial_{x_j} u ={}& \sum_{j,k=1}^d \partial_{x_j} \brac{D_{j k}(u) \partial_{x_k} u},
    \end{aligned}
\end{equation}
where $u=u(x, t)$ is the unknown $m$-vector valued function of time $t\geq 0$ and space coordinate $x=(x_1, \ldots, x_d)\in \mathbb{R}^d$, $a_j = a_j(u)$ and $D_{jk} = D_{jk}(u)$ $(j, k=1,\ldots, d)$ are given $m\times m$-matrix valued smooth functions of $u\in \Omega_u$ with state space $\Omega_u$ open and convex. 
Such systems frequently appear in non-equilibrium fluid dynamics, radiative transfer, traffic flow, and semiconductor devices. 
%
To efficiently solve \eqref{equ:hp-equ}, a variety of relaxation schemes have been proposed in the literature \cite{JPT1998,JP2000,NP2000,BR2013,BPR2013,Boscarino2014HighSisc,chen2023relaxation,Aregba2001diffusion,aregba2004explicit,lund2012siam}. The construction of these schemes consists of two steps: first a relaxation model is formulated to approximate \eqref{equ:hp-equ} and then a suitable scheme is carefully designed to solve it. In this methodology, the validity of the relaxation model approximating the given system is crucial, otherwise the resulting numerical scheme is invalid.

Here we briefly introduce some relaxation models for systems of the form \eqref{equ:hp-equ}. The first model is the hyperbolic system with diffusive relaxation, which was considered as an approximation of the convection-diffusion equation (CDE) \cite{JPT1998,JP2000,NP2000,BR2013,BPR2013,Boscarino2014HighSisc}. %
Based on this relaxation model and implicit-explicit (IMEX) Runge-Kutta schemes, general approaches were presented in \cite{BR2013,BPR2013,Boscarino2014HighSisc} to overcome the classical parabolic time-step restriction $\Delta t = O( \Delta x^2)$ ($\Delta t$ is the time step and $\Delta x$ is the mesh size) for CDEs.
By relaxing both the convective fluxes and the diffusive fluxes, a relaxation model was recently introduced in \cite{chen2023relaxation} for the entropy dissipative system of viscous conservation laws. With this model, relaxation schemes in the finite volume framework are also developed to avoid the parabolic restriction \cite{chen2023relaxation}. 
In \cite{cavalli2012discontinuous}, a semilinear hyperbolic system is constructed to approximate nonlinear diffusion equations by extending the idea of Jin-Xin relaxation model for hyperbolic conservation laws. 
In this system, an additional auxiliary variable is introduced for the common relaxation model so that the convective part of the final system is linear. Thanks to this advantage,  the numerical schemes obtained in \cite{cavalli2012discontinuous,Ca2013} do not require solving implicit nonlinear problems. 

Note that the relaxation models mentioned above are direct relaxation for the convective flux or the diffusive flux of the given system. 
Another popular class of relaxation systems are the kinetic models, which consist of linear advection of distribution functions and stiff collision terms that derive the distribution functions tend to the equilibrium in the relaxation limit.  As for systems of the form \eqref{equ:hp-equ}, we mention the lattice Boltzmann equation (LBE) for the CDEs \cite{ginzburg2007lattice,ginzburg2012truncation,shi2009lattice,chai2013lattice,Yoshida2010Multiple,Huang2015boundary,zhang2019lattice} and the diffusive kinetic model for nonlinear hyperbolic-parabolic systems \cite{Aregba2001diffusion,aregba2004explicit}.  In particular, the LBE is the basis of the popular lattice Boltzmann method (LBM), which is a powerful solver for fluid dynamics equations and the CDEs \cite{benzi1992lattice,chen1998lattice,aidun2010lattice}.

Though the relaxation-based numerical methods exhibit high computational efficiency compared to the traditional ones in terms of stability and simplicity, justification of the approximation validity for the underlying relaxation models is far from complete. 
%
On this point, we remark that for the relaxation model in \cite{JPT1998,JP2000,NP2000,BR2013,BPR2013,Boscarino2014HighSisc}, convergence of its solution to that of a nonlinear CDE is proved with initial data prescribed around a traveling wave solution \cite{jin1998Diffusion}. For the scalar case of the kinetic approximation in \cite{Aregba2001diffusion,aregba2004explicit}, convergence proof is established with kinetic entropy inequalities in \cite{BGN2001}. 
In \cite{LN2002}, a rigorous result of convergence toward the formal limit of a class of BGK approximations of parabolic systems in one space dimension is presented.
The diffusive relaxation limit of the Jin-Xin system toward viscous conservation laws in the multi-dimensional setting is studied in \cite{CS2023diffusive}.
Nevertheless, for most of the relaxation models arising from numerical schemes for \eqref{equ:2-1-pde}, their validity is still to be clarified.

Actually, the general form of the relaxation models in developing the above numerical schemes can be written as
\begin{equation}\label{equ:2-1-pde}
	\begin{aligned}
		\partial_t U + \frac{1}{\varepsilon} \sum^d_{j=1} A_j(U;\varepsilon)\partial_{x_j} U &= \frac{1}{\varepsilon^2}Q(U;\varepsilon),
	\end{aligned}
\end{equation}
which is a set of first-order PDEs with a small parameter $\varepsilon >0 $.
Here $U =(u, ~w)^T$ is the $n$-vector function with $u\in \mathbb{R}^{n-r}(m=n-r)$, $ w\in \mathbb{R}^{r}$ of $(x, t)$, taking values in the state space $\Omega$ (an open subset  of $R^n $); $A_j(U;\varepsilon)$ and $Q(U;\varepsilon)$ are the respective $n\times n$-matrix and $n$-vector smooth functions of $(U;\varepsilon)\in \Omega\times (0, 1]$.
Additionally, $w$ denotes the vector of relaxed variables, and $Q(U;\varepsilon)$ encodes the stiff source terms driving $U$ toward equilibrium. 
For the case of $A_j(U;\varepsilon) = \hat A_j(\varepsilon U) + \varepsilon \bar{A}_j(U)$ and $\varepsilon$-independent source term $Q(U;\varepsilon) = Q(U)$, a framework for analyzing singular limits is presented in \cite{lattanzio2001hyperbolic} based on structural stability conditions for hyperbolic relaxation systems \cite{yong1999singular}. 
This result is further improved in \cite{peng2016parabolic} for a more general source term 
$Q(U;\varepsilon)$ depending on $\varepsilon$ and a more general compatibility conditions.
Using these as stepping-stones, convergence of the complete problem \eqref{equ:2-1-pde} has recently been resolved in \cite{peng2025convergence} with a unified energy estimate. 
Thus we can employ these results to validate the aforementioned relaxation models.

The main contributions of this work are two manifolds. First, we justify the convergence of five representative relaxation models mentioned above to the target hyperbolic-parabolic system by verifying the convergence criteria in \cite{yong1999singular,lattanzio2001hyperbolic,peng2025convergence}. 
These include three relaxation models via direct relaxation for the convective or diffusive flux from \cite{JPT1998,JP2000,NP2000,BR2013,BPR2013,Boscarino2014HighSisc,chen2023relaxation,cavalli2012discontinuous}, the LBE for CDEs  \cite{ginzburg2007lattice,ginzburg2012truncation,shi2009lattice,chai2013lattice,Yoshida2010Multiple,Huang2015boundary,zhang2019lattice} and the kinetic approximation for nonlinear hyperbolic-parabolic systems proposed in \cite{aregba2004explicit}. 
Since the relaxation models are the basis of the corresponding numerical schemes, the validity of these models is of fundamental importance, particularly for their convergence and consistency analysis. 
On the other hand, we propose a new relaxation model for the general multi-dimensional hyperbolic-parabolic system of the form \eqref{equ:hp-equ}. By some mild assumptions on the system, we show that the proposed relaxation model satisfies the convergence criteria in \cite{yong1999singular,lattanzio2001hyperbolic,peng2025convergence} and thus its validity as an approximation of \eqref{equ:hp-equ} is guaranteed. Notice that the existing five relaxation models we analyze are constructed only for special case of \eqref{equ:hp-equ}, while our new relaxation model is the first one valid for the general system. We believe that the new model is promising for broader applications in multiscale modeling and developing efficient numerical schemes for complex dissipative systems.

The rest of the paper is organized as follows. Section \ref{sec2} introduces the convergence criteria in \cite{yong1999singular,lattanzio2001hyperbolic,peng2025convergence}.  The verification of these criteria for five relaxation models is presented in Section \ref{sec3}. 
In Section \ref{sec4} we propose a new relaxation model and justify its validity. Finally some conclusions and remarks are given in Section \ref{sec5}.

\section{Convergence criteria}
\label{sec2}

In this section, we give the convergence criteria proposed in \cite{yong1999singular,lattanzio2001hyperbolic,peng2025convergence}, which guarantee that the solution of the hyperbolic relaxation system \eqref{equ:2-1-pde} converges to that of the limit equation, \ie, the hyperbolic-parabolic system \eqref{4eq:u0-equ} below, in the relaxation limit $\varepsilon \rightarrow 0$. Denote by $\Omega$ and $\Omega_u$ the state space of solutions to \eqref{equ:2-1-pde} and \eqref{4eq:u0-equ}, respectively.
The criteria consist of five conditions on the relaxation system \eqref{equ:2-1-pde} \cite{yong1999singular,lattanzio2001hyperbolic,peng2025convergence}:
\begin{enumerate}[label=(\roman*)]
    \item \label{cond:2-ssc-1} $Q(U;\varepsilon)$ can be written as $Q(U;\varepsilon) = (0, ~q(U;\varepsilon))^T$
with $q: \Omega\times (0, 1] \rightarrow \mathbb{R}^r $ a smooth function. 
    \item \label{cond:2-ssc-2} $q(u, w;0) = 0$  if and only if $ w=0$; 
    and $ \partial_w q(u, 0;0)$ is invertible for all $u\in \Omega_u$. 
    \item  \label{cond:2-ssc-3}	 The system is symmetrizable hyperbolic, \ie, there exists a symmetric positive definite matrix $A_0(U;\varepsilon)$ (called symmetrizer) such that
    \begin{equation}\nonumber		A_0(U;\varepsilon)A_j(U;\varepsilon)=A_j(U;\varepsilon)^TA_0(U;\varepsilon), \quad  \forall U \in \Omega.
\end{equation}
\item \label{cond:2-ssc-4} There exists a symmetric positive definite $r \times r$ matrix $S(u)$ such that
\begin{equation}\nonumber
A_0(u,0;0)
\partial_U Q(u,0;0)
+(\partial_U Q(u,0;0))^TA_0(u,0;0)
\leq	
		-\begin{pmatrix}
				0 & 0\\
				0 & S(u)
		\end{pmatrix},
        \quad  \forall u \in \Omega_u.
\end{equation}
\item \label{cond:2-ssc-5} $A_j^{11}(u,0;0)=0$ and $ \partial_u A_{j}^{11}(u,0;0) = 0$ for all $u\in \Omega_u$. 
\end{enumerate}	
In the above equations, the superscript $T$ denotes the transpose operator.
Throughout the paper, a vector $U \in \mathbb{R}^n$ is frequently decomposed as $\begin{pmatrix}U^I\\U^{II}\end{pmatrix}$, where \( U^I \in \mathbb{R}^{n-r} \) and \( U^{II} \in \mathbb{R}^r \) denote the subvectors composed of the first $n-r$ and last $r $ components of $V$, respectively. Similarly, an $n \times n$  matrix  $A$ is written as $\begin{pmatrix}A^{11}& A^{12}\\A^{21}&A^{22}\end{pmatrix}$, where each block corresponds to a natural partition of rows and columns into the first $n-r$ and last $r$ components.

Some we give some explanations on the above conditions. Condition \ref{cond:2-ssc-1} gives the usual form of source term for first-order hyperbolic problems with relaxation. Condition \ref{cond:2-ssc-2} is the parabolic structural assumption \cite{lattanzio2001hyperbolic}. Condition \ref{cond:2-ssc-3} means the symmetrizable hyperbolicity of the system of first-order partial differential equations. Condition \ref{cond:2-ssc-4} stands for the partial dissipation property of system \eqref{equ:2-1-pde}. Condition \ref{cond:2-ssc-5} guarantees the compatibility of the relaxation system with the limit hyperbolic-parabolic system. 

It should be noted that Conditions \ref{cond:2-ssc-1}--\ref{cond:2-ssc-4} are the structural stability conditions proposed in \cite{yong1999singular}, which guarantee the convergence of solutions of relaxation systems to those of hyperbolic systems. These conditions have been tacitly respected by many well-known physical models \cite{yong1999singular,yong2001,Yong2008}. When combined with the compatibility condition \ref{cond:2-ssc-5}, they ensure the convergence of solutions of the relaxation models to those of the limit hyperbolic-parabolic systems \cite{peng2025convergence}.

To determine the limit equation of \eqref{equ:2-1-pde}, we expand its solution in terms of $\varepsilon$ as 
\begin{equation}
    u \sim u_0 + \varepsilon u_1 + \varepsilon^2 u_2 + \ldots, \quad w \sim w_0 + \varepsilon w_1 + \varepsilon^2 w_2 + \ldots.
\end{equation}
Substituting this Ansatz into \eqref{equ:2-1-pde} and 
equating the coefficients of different powers of $\varepsilon$, we obtain 
\begin{subequations}
\begin{align}
   & \varepsilon^{-2}: {} w_0 = 0,\\
    &\varepsilon^{-1}: {} w_1 = {} \partial_w q(u_0, 0;0)^{-1}\bigg[\sum_{j=1}^d A_j^{21}(u_0, 0;0)\partial_{x_j} u_0 
    - \partial_\varepsilon q(u_0, 0;0)\bigg],\label{equ:u_0-1-b}\\
   & \label{equ:u_0-1-c} \varepsilon^0 : {}\partial_t u_0 
   +  \sum_{j=1}^d A^{12}_j(u_0, 0;0)\partial_{x_j} w_1 
   + \sum_{j=1}^d w_1 \partial_w A^{11}_{j}(u_0, 0;0) \partial_{x_j} u_0 \\
   & \nonumber \hspace{7mm} + \sum_{j=1}^d \partial_\varepsilon A^{11}_{j}(u_0, 0;0)\partial_{x_j} u_0  =0, 
  \end{align}
\end{subequations}
where Conditions \ref{cond:2-ssc-1}, \ref{cond:2-ssc-2} and \ref{cond:2-ssc-5} have bee used.
Combining \eqref{equ:u_0-1-b} and \eqref{equ:u_0-1-c} gives the limit equation
%
\begin{equation}\label{4eq:u0-equ}
\begin{aligned}
    &\partial_t u_0 - \sum_{j=1}^d A^{12}_j(u_0, 0;0) \partial_u \brac{ \partial_w q(u_0, 0;0)^{-1} \partial_\varepsilon q(u_0, 0;0) }\partial_{x_j} u_0\\
    {}&-
    \partial_w q(u_0, 0;0)^{-1} \partial_\varepsilon q(u_0, 0;0)\sum_{j=1}^d \partial_w A^{11}_{j}(u_0, 0;0) \partial_{x_j} u_0 + \sum_{j=1}^d \partial_\varepsilon A^{11}_{j}(u_0, 0;0)\partial_{x_j} u_0\\
    {}&+ 
    \sum_{j, k=1}^d A^{12}_j(u_0, 0;0) \partial_w q(u_0, 0;0)^{-1}A_k^{21}(u_0, 0;0)\partial_{x_j}\partial_{x_k} u_0 \\
    {}&+ 
    \sum_{j, k=1}^d A^{12}_j(u_0, 0;0)\partial_u \brac{\partial_w q(u_0, 0;0)^{-1}  A_k^{21}(u_0, 0;0)}\partial_{x_j} u_0\partial_{x_k} u_0\\
    {}&+ 
    \sum_{j, k=1}^d \partial_w q(u_0, 0;0)^{-1} A_k^{21}(u_0, 0;0)\partial_{x_k} u_0\partial_w A^{11}_{j}(u_0, 0;0) \partial_{x_j} u_0 = 0.
\end{aligned}    
\end{equation}
Denote by $U^\varepsilon = (u^\varepsilon, w^\varepsilon)$ the solution to system \eqref{equ:2-1-pde}.
Let $s>\frac{d}{2}+1$ be an integer and denote $\norm{\cdot}_s$ as the norm of the Sobolev
space $H^s$.
Under structural stability conditions \ref{cond:2-ssc-1}-\ref{cond:2-ssc-4} and the compatibility condition \ref{cond:2-ssc-5}, the following convergence result can be found in \cite{peng2025convergence}:

\begin{theorem}\label{thm:convergence}
Denote  $\tilde{U}(x, \varepsilon)=(\tilde{U}^{I}(x, \varepsilon),\tilde{U}^{II}(x, \varepsilon))^T$ and $u_0(x, 0)$ as the initial data for the relaxation system \eqref{equ:2-1-pde} and the limit equation \eqref{4eq:u0-equ}, respectively, which are assumed to satisfy the consistency condition
\begin{equation*}
		\norm{\tilde{U}^{I}(\cdot, \varepsilon) - u_0(\cdot,0)}_s = O(\varepsilon).
\end{equation*}
Assume Conditions \ref{cond:2-ssc-1}-\ref{cond:2-ssc-5} hold.
Then there exists a positive constant $T_\star$, independent of $\varepsilon$, such that system \eqref{equ:2-1-pde} with the initial data has a unique solution in $C ( [0, T_\star], H^s )$. 

If we further assume that the limit equation \eqref{4eq:u0-equ} has a unique solution $u_0(x ,t) \in C([0,T_\star],H^{s+1})$ $\bigcap C^1([0,T_\star],H^s)$, then there exist a positive constant $K(T_\star)$ such that 
\begin{equation*}
		\sup _{t \in [0, T_\star ]}\norm{u^\varepsilon(t) - u_0(t)}_s \leq K(T_\star) \varepsilon
\end{equation*}
for sufficiently small $\varepsilon$.
\end{theorem}

\begin{remark}
    Theorem \ref{thm:convergence} ensures that solution of the relation system \eqref{equ:2-1-pde} converges to that of the limit equation \eqref{4eq:u0-equ} under Conditions \ref{cond:2-ssc-1}-\ref{cond:2-ssc-5}. 
    For our purpose of validating the relaxation system for a target given system  \eqref{equ:hp-equ}, we also need to verify that the limit equation \eqref{4eq:u0-equ} admitted by $u_0$ is exactly the same as the target hyperbolic-parabolic system \eqref{equ:hp-equ} we consider.
\end{remark}

\begin{remark}
Though there seems to be no theory on the existence of solution to the general equation  \eqref{4eq:u0-equ}, for the specific target equations \eqref{equ:hp-equ} considered in work, which are all symmetrizable hyperbolic-parabolic systems, they admit a unique well-defined local-in-time solution according to the local existence theory in \cite{kawashima1984systems} (Theorem 2.9). For general local existence results on hyperbolic-parabolic systems, we refer the readers to \cite{ladyv1968linear,taylor1997partial}.
\end{remark}

\section{Validity of relaxation models}
\label{sec3}

In this section, we justify the validity of five relaxation models arising from numerical schemes for the hyperbolic-parabolic system \eqref{equ:hp-equ} by verifying the convergence criteria \ref{cond:2-ssc-1}-\ref{cond:2-ssc-5}.
Each of these models is an approximation of a special case of \eqref{equ:hp-equ}.

\subsection{Relaxation model for 1D CDE}
\label{sec3.1}

The first relaxation model we consider is to approximate the 1D CDE
\begin{equation}\label{eg1:equ-1}
    \partial_t u + \partial_x f(u) = \partial_{xx} b(u),
\end{equation}
where $u=u(t, x)$ is the unknown function and $\partial_u b(u)>0$.
To overcome the time-step restriction $\Delta t = O( \Delta x^2)$ for explicit method of the CDE \eqref{eg1:equ-1}, the following hyperbolic system with diffusive relaxation  was formulated as a basis for developing efficient numerical schemes \cite{JPT1998,JP2000,NP2000,BR2013,BPR2013,Boscarino2014HighSisc}:
\begin{equation}\label{eg1:equ-2}
\begin{aligned}
    &\partial_t u+ \partial_x v=0, \\
    &\varepsilon^2 \partial_t v+ \partial_x b(u) = -v + f(u).
\end{aligned}
\end{equation}
Here $v=v(t, x)$ is a dissipative variable approximating $f(u)-\partial_x b(u)$. For this system we have $q(U;\varepsilon)=-v + f(u)$ according to the notation in Condition \ref{cond:2-ssc-1}. Then $q(U;0)=0$ if and only if $v=f(u)$, which violates Condition \ref{cond:2-ssc-2}.
Thus instead of using the original form \eqref{eg1:equ-2} to verify convergence criteria, we introduce $w = \varepsilon v$ and rewrite the system as
\begin{equation}\label{n33}
\begin{aligned}
    &\partial_t u+ \frac{1}{\varepsilon} \partial_x w =0, \\
    &\partial_t w + \frac{1}{\varepsilon}\partial_x b(u) = -\frac{w}{\varepsilon^2} + \frac{f(u)}{\varepsilon}.
\end{aligned}
\end{equation}
It is easy to see that the limit equation for $u$ in \eqref{n33} is exactly the CDE \eqref{eg1:equ-1}.

Next we show that \eqref{n33} admits Conditions \ref{cond:2-ssc-1}-\ref{cond:2-ssc-5}. 
Note that Condition \ref{cond:2-ssc-1} is obviously true and Condition \ref{cond:2-ssc-2} holds since $q(U;\varepsilon)=-w + \varepsilon f(u)$ for the transformed system  \eqref{n33}. To verify the other conditions, we set $U=(u, ~w)^T$ and rewrite \eqref{n33} as
\begin{equation}\nonumber
    \partial_t U + \frac{1}{\varepsilon}A(U;\varepsilon)\partial_x U = \frac{1}{\varepsilon^2} Q(U;\varepsilon)
\end{equation}
with 
\begin{equation}\nonumber
\begin{aligned}
    A(U;\varepsilon) = \begin{pmatrix}
        0&1\\
        \partial_u b(u) & 0
    \end{pmatrix},
    \quad
   Q(U;\varepsilon)
   :=
   \begin{pmatrix}
        0\\        q(u,w;\varepsilon)
    \end{pmatrix}
   = 
   \begin{pmatrix}
        0\\
        -w + \varepsilon f(u)
    \end{pmatrix}. 
\end{aligned}    
\end{equation}
With this form we see that $A^{11}(U;\varepsilon)=0$ and Condition \ref{cond:2-ssc-5} is admitted.
Moreover, we define a symmetric matrix
\begin{equation}\nonumber
\begin{aligned}
    A_0(U;\varepsilon) = \begin{pmatrix}
        \partial_u b(u) & \varepsilon \partial_u f(u)\\
        \varepsilon \partial_u f(u) & 1
    \end{pmatrix},
\end{aligned}
\end{equation}
which is positive definite if 
\begin{equation}\nonumber
    \varepsilon \abs{\partial_u f(u)} < \sqrt{\partial_u b(u)}.
\end{equation}
This is the so-called subcharacteristic condition \cite{liu1987hyperbolic} and is satisfied for sufficiently small $\varepsilon$. 
It is easy to check that $A_0(U;\varepsilon)A(U;\varepsilon)$ is symmetric and $A_0(u,0;0)\partial_U Q(u,0;0) = \diag\{0, -1\}$, which leads to the satisfaction of Conditions \ref{cond:2-ssc-3} and \ref{cond:2-ssc-4}. Thus we have verified all the conditions and the validity of the relaxation model \eqref{eg1:equ-2} as an approximation of \eqref{eg1:equ-1} is guaranteed.

\subsection{Relaxation model for viscous conservation law}\label{sec3.2}

The second relaxation model is to approximate the multi-dimensional viscous conservation law \cite{chen2023relaxation}
\begin{equation}\label{equ:examp-312-limit}
\begin{aligned}
   \partial_t u + \sum_{j=1}^d \partial_{x_j} f_j(u) 
   = \sum_{j=1}^d \partial_{x_j}
   \brac{ \sum_{k=1}^d B_{jk}(u) \partial_{x_k} u},
\end{aligned}    
\end{equation}
where $u=u(x, t)$ is the unknown $n$-vector valued function of $(x, t)=(x_1,\cdots,x_d, t)\in \mathbb{R}^d\times [0, +\infty)$, $f_j=f_j(u), j=1,\ldots, d$ are $n$-vector valued flux functions of $u$, and $B_{jk}=B_{jk}(u), j,k=1,2,\ldots, d$ are $n\times n$-matrix valued functions. 
According to \cite{chen2023relaxation}, system \eqref{equ:examp-312-limit} is assumed to satisfy 
\begin{enumerate}[label=(A\arabic*)]
\item \label{assm:6-1} There exists a convex entropy function $\eta(u)$ together with associated entropy fluxes $g_j=g_j(u), j=1,2,\dots,d$ such that 
$  \eta_{uu}(u)>0, \quad \partial_u g_{j} = \eta_{u} \partial_u f_{j}, \quad j=1,2,...,d.$
 \item \label{assm:6-2} The $nd\times nd$ matrix matrix $\diag\{\eta_{uu}, \ldots, \eta_{uu}\}B$ is symmetric positive definite, where $B$ is defined as
\begin{equation}\nonumber
    B = \begin{pmatrix}
        B_{11}(u) & \ldots & B_{1d}(u) \\
        \vdots & \vdots & \vdots\\
        B_{d1}(u) & \ldots & B_{dd}(u) \\
    \end{pmatrix}\in \mathbb{R}^{nd\times nd}.
\end{equation}
\end{enumerate}
Under the above assumptions, $B$ is invertible and its eigenvalues are all positive. Additionally, it holds the entropy inequality 
\begin{equation}\nonumber
    \partial_t \eta  + \sum_{j=1}^d \partial_{x_j} g_j(u) - \sum_{j, k=1}^d \partial_{x_j}\brac{  \eta_u B_{jk}\partial_{x_k} u} = -\sum_{j, k=1}^d  \eta_{uu}B_{jk}\partial_{x_j} u \partial_{x_k} u \leq 0.
\end{equation}
Here we remark that the positive definiteness assumption for the matrix $\diag\{\eta_{uu}, \ldots, \eta_{uu}\}B$ is for the following verification of convergence criteria, while this matrix is only required to be positive semi-definite for numerical computations in \cite{chen2023relaxation}.  

By relaxing both the convective and diffusive fluxes and extending the ideal of Jin-Xin relaxation model for hyperbolic conservation law \cite{jin1995relaxation}, a hyperbolic relaxation system is developed in \cite{chen2023relaxation} for \eqref{equ:examp-312-limit} (Here we reformulate the original model in \cite{chen2023relaxation} with $w_i = \varepsilon v_i$.):
\begin{equation}
\begin{aligned}\label{equ:6-2}
    &\partial_t u+\frac{1}{\varepsilon} \sum_{j=1}^d \partial_{x_j} w_j =0, \\
    &\partial_t w_i +\frac{1}{\varepsilon} \sum_{j=1}^d  \brac{B_{ij}(u)+\varepsilon^2 a^2 I_n} \partial_{x_j} u =\frac{1}{\varepsilon^2}(\varepsilon f_i(u)- w_i ), \quad i=1,\dots, d,
\end{aligned}
\end{equation}
with $w_i, i=1,2, \ldots, d \in \mathbb R^{n}$ being relaxation variables.
Denoting $U=(u, w_1, \dots, w_d)\in \mathbb{R}^{n+nd}$, we rewrite the system \eqref{equ:6-2} as 
\begin{equation}\label{nn310}
\begin{split}
&    \partial_t U + \frac{1}{\varepsilon}\sum_{j=1}^d A_j(U; \varepsilon)\partial_{x_j} U =\frac{1}{\varepsilon^2} Q(U;\varepsilon),\\
& Q(U;\varepsilon) = (0, ~\varepsilon f_1(u)- w_1, \dots, ~\varepsilon f_d(u)- w_d)^T \in \mathbb{R}^{n+nd},\\
& A_j(U; \varepsilon) 
 =  \begin{pmatrix}
        O_{n\times n} & I_n\delta_{j1} & \dots & I_n \delta_{jd} \\
        B_{1j}(u)+\varepsilon^2 a^2 I_n & O_{n\times n} & \dots & O_{n\times n}\\
        \vdots & \vdots & \vdots & \vdots \\
        B_{dj}(u)+\varepsilon^2 a^2 I_n & O_{n\times n} & \dots & O_{n\times n}
    \end{pmatrix}
    \in \mathbb{R}^{n(1+d)\times n(1+d)}, \quad j=1, \dots, d.
\end{split}
\end{equation}
Here $\delta_{jk}$ denotes the delta function, $I_n$ the $n\times n$ identity matrix and $O_{n\times n}$ the $n \times n$ zero matrix.

Now we verify Conditions \ref{cond:2-ssc-1}-\ref{cond:2-ssc-5}. Note that Conditions \ref{cond:2-ssc-1}, \ref{cond:2-ssc-2} and \ref{cond:2-ssc-5} are obviously true.
Define 
\begin{equation}\nonumber
    H= B+ \varepsilon^2 a^2 \begin{pmatrix}
        I_n & \ldots & I_n\\
        \vdots & \vdots & \vdots\\
        I_n & \ldots & I_n \\ 
    \end{pmatrix}
\end{equation}
and compute
\begin{equation*}
    \diag\{\eta_{uu}, \ldots, \eta_{uu}\} H
    = \diag\{\eta_{uu}, \ldots, \eta_{uu}\} B
    + \varepsilon^2 a^2 \begin{pmatrix}
        \eta_{uu} & \ldots & \eta_{uu}\\
        \vdots & \vdots & \vdots\\
        \eta_{uu} & \ldots & \eta_{uu} \\ 
    \end{pmatrix},
\end{equation*}
which is positive definite according to Assumption \ref{assm:6-2}. Thus $H$ is invertible and we define the symmetrizer as
\begin{equation}\nonumber
    A_0 = \diag\{\eta_{uu},~ \diag\{\eta_{uu}, \ldots, \eta_{uu} \} H^{-1}\},
\end{equation}
the positive definiteness of which is due to that of
\begin{equation*}
H^T [\diag\{\eta_{uu}, \ldots, \eta_{uu} \} H^{-1}]H =\brac{\diag\{\eta_{uu}, \ldots, \eta_{uu} \} H}^T.
\end{equation*}
With the above matrix $H$, the matrices  $A_j(U;\varepsilon),j=1,2,...,d$ in \eqref{nn310} can be written as 
\begin{equation*}
    A_j(U;\varepsilon) = \begin{pmatrix}
        O_{n\times n} &\begin{pmatrix}
            I_n \delta_{j1} & ... & I_n\delta_{jd}
        \end{pmatrix} \\
        H \begin{pmatrix}
            I_n\delta_{j1} & ... & I_n\delta_{jd}
        \end{pmatrix}^T & O_{nd\times nd}
    \end{pmatrix}, \quad
    j=1,2,...,d.
\end{equation*}
Then we have
\begin{equation*}
    \begin{split}
        A_0 A_j 
        &
        = 
        \begin{pmatrix}
        O_{n\times n} &  \begin{matrix}
            (\eta_{uu}\delta_{j1} & ... & \eta_{uu}\delta_{jd})
        \end{matrix} \\
        \diag\{\eta_{uu}, \ldots, \eta_{uu} \} H^{-1} H \begin{pmatrix}
            I_n\delta_{j1} & ... & I_n\delta_{jd}
        \end{pmatrix}^T & O_{nd\times nd}
    \end{pmatrix} \\
    &
    =
     \begin{pmatrix}
        O_{n\times n} &  \begin{pmatrix}
            \eta_{uu}\delta_{j1} & ... & \eta_{uu}\delta_{jd}
        \end{pmatrix} \\
        \begin{pmatrix}
            \eta_{uu}\delta_{j1} & ... & \eta_{uu}\delta_{jd}
        \end{pmatrix}^T & O_{nd\times nd}
    \end{pmatrix}, 
    \end{split}
\end{equation*}
which is symmetric and leads to the satisfaction of Condition \ref{cond:2-ssc-3}. Finally, Conditions \ref{cond:2-ssc-4} holds since
\begin{equation*}
    A_0(u,0;0)\partial_U Q(u,0;0)=\diag\{0,~ -\diag\{\eta_{uu}, \ldots, \eta_{uu} \} B\}
    \leq 0.
\end{equation*}
Thus the validity of the relaxation model \eqref{equ:6-2} for the viscous conservation law \eqref{equ:examp-312-limit} is guaranteed.

\subsection{Relaxation model for multidimensional nonlinear diffusion equation}

In this example, we consider a relaxation model for the nonlinear diffusion equation 
\begin{equation}\label{equ:eg3-3-1}
    \partial_t u - \Delta p(u) = 0,
\end{equation}
where $u=u(x, t)$ is an unknown function and $p(u)$ is a given function satisfying $\partial_u p(u) > 0$.
This class of PDEs appears in filtration, phase transition, biochemistry, image analysis, and dynamics of biological groups  \cite{Aronson1986,PeronaMalik1990,OkuboLevin2001,Vazquez2007}.
In \cite{cavalli2012discontinuous}, a semilinear hyperbolic system is constructed to approximate \eqref{equ:eg3-3-1} by extending the idea of Jin-Xin relaxation model for hyperbolic conservation laws \cite{jin1995relaxation}. 
Specifically, an additional auxiliary variable is introduced for the common relaxation model so that the convective part of the final system is linear. Thanks to this advantage,  the resulting numerical schemes based on the relaxation model do not require solving implicit nonlinear problems \cite{cavalli2012discontinuous}.

The relaxation model in \cite{cavalli2012discontinuous} for \eqref{equ:eg3-3-1} reads as 
\begin{equation}\label{n310}
\begin{aligned}
    &\partial_t u +\nabla \cdot \hat v =0, \\
    &\partial_t \hat v + \frac{1}{\varepsilon^2} \nabla \hat w = -\frac{ \hat v}{\varepsilon^2}, \\
    &\partial_t \hat w + a^2 \nabla \cdot \hat v
    =\frac{1}{\varepsilon^2}(p(u)-\hat w),
\end{aligned}
\end{equation}
where $\hat v =(\hat v_1(x, t), \dots, \hat v_d(x, t))$ and $\hat w = \hat w(x, t)$ are two auxiliary variables and $a$ is a constant satisfying the subcharacteristic condition
\begin{equation}\label{n311}
    -a < \partial_u p(u) < a. 
\end{equation}
It is easy to see that the limit equation of $u$ in \eqref{n310} is the target equation \eqref{equ:eg3-3-1}.
To check the convergence criteria, we set $v = \varepsilon \hat v$ and $w= \hat w - p(u)$ and transform the system \eqref{n310} to
\begin{equation}\nonumber
\begin{aligned}
    &\partial_t u + \frac{1}{\varepsilon} \nabla \cdot v=0, \\
    &\partial_t v + \frac{1}{\varepsilon} \nabla( w + p(u)) =-\frac{1}{\varepsilon^2}v, \\
    &\partial_t w + \frac{1}{\varepsilon} (a^2-\partial_u p(u)) \nabla \cdot v=-\frac{1}{\varepsilon^2}w.
\end{aligned}
\end{equation}
With notation $U = (u, v_1, \dots, v_d, w)^T$, the above system can be written as 
\begin{equation}\nonumber
\begin{split}
&    \partial_t U + \frac{1}{\varepsilon}\sum_{j=1}^d A_j(U)\partial_{x_j} U =\frac{1}{\varepsilon^2} Q(U),\\
& A_j(U) 
 = \begin{pmatrix}
        0 & e_j^T & 0\\[2mm]
        \partial_u p(u)e_j & O_{d\times d} & e_j \\
        0 & (a^2-\partial_u p(u)) e_j^T & 0 
    \end{pmatrix}, ~j=1, \dots, d, 
\quad Q(U) = \begin{pmatrix}
        0 \\ -v \\ -w
    \end{pmatrix}.
\end{split}
\end{equation}
Here $e_j$ is the $j$-th vector of the canonical basis in $\mathbb{R}^d$. 
With this form, one can easily see that Conditions \ref{cond:2-ssc-1}, \ref{cond:2-ssc-2} and \ref{cond:2-ssc-5} are true. 
Moreover, since $0 < \partial_u p(u) < a^2$ according to \eqref{n311}, we define a symmetrizer 
$
   A_0 = \diag\{p'(u), ~I_d, ~(a^2-\partial_u p(u))^{-1}\}.
$
It is straightforward to compute that $A_0A$ is symmetric and 
$$
A_0(u,0;0) \partial_U Q(u,0;0)=\diag\{0, ~-I_d, ~-(a^2-\partial_u p(u))^{-1}\} \leq 0.
$$
Thus Conditions \ref{cond:2-ssc-3} and \ref{cond:2-ssc-4} are also true and the validity of the relaxation model \eqref{n310} approximating \eqref{equ:eg3-3-1} is guaranteed.

\subsection{Lattice Boltzmann model for CDE}

The above relaxation models are all constructed via direct relaxation for the convective or diffusive flux of the given system. 
In the following we consider two kinetic approximation models, which are relaxation systems for a set of distribution functions. 
Specifically, we first consider a lattice Boltzmann model for the nonlinear CDE in two dimensions:
\begin{equation}\label{equ:lbe-1}
    \begin{aligned}
        \partial_t u + \nabla\cdot f(u) = \nabla \cdot (D(u) \nabla u),
    \end{aligned}
\end{equation}
where $u=u(x,t) \in \mathbb R^2$, $f = f(u)$ and $D=D(u)$ are the given flux function and diffusion coefficient of $u$.
In the literature, \eqref{equ:lbe-1} has been frequently solved with the mesoscopic LBM \cite{ginzburg2007lattice,ginzburg2012truncation,shi2009lattice,chai2013lattice,Yoshida2010Multiple,Huang2015boundary,zhang2019lattice}, which is a direct discretization of the continuous lattice Boltzmann model. In contrast to the wide application of the LBM for \eqref{equ:lbe-1}, rigorous justification of the corresponding lattice Boltzmann model is rare. Here we address this problem by verifying the convergence criteria in Section \ref{sec2}.

The lattice Boltzmann model for the 2D CDE \eqref{equ:lbe-1} reads as \cite{ginzburg2007lattice,ginzburg2012truncation,shi2009lattice,chai2013lattice,Yoshida2010Multiple,Huang2015boundary,zhang2019lattice}
\begin{equation}\label{equ:lbe-2}
    \partial_t g_i + \frac{1}{\varepsilon} \xi_i \cdot \nabla g_i = \frac{1}{\varepsilon^2} \frac{1}{\tau}(g_i^{(eq)} - g_i), \quad i=1,\ldots, 5,
\end{equation}
where $g_i=g_i(x,t)$ is the distribution function in the $i$-th direction; 
$\varepsilon>0$ is a small parameter; the discrete velocities are ${\xi}_1= (0,0)^T$, ${\xi}_2 =-{\xi}_4 =(1,0)^T$, ${\xi}_3=-{\xi}_5=(0,1)^T$; 
$\tau=\tau(u)$ is the relaxation time; and $g^{(eq)}_i$ is the equilibrium distribution function given by
\begin{equation}\nonumber
    g^{(eq)}_i = \omega_i u + 3\varepsilon \omega_i \xi_i \cdot f(u)
\end{equation}
with $\omega_0 = \frac{1}{3}$, $\omega_{1,2,3,4} = \frac{1}{6}$ the weight coefficients.
The macroscopic variable $u$ is defined via distribution functions as
\begin{equation}\nonumber
    u=\sum_{i=1}^5 g_i.
\end{equation}

We first derive the limit equation of \eqref{equ:lbe-2} as $\varepsilon \rightarrow 0$. To do this, we sum up equations \eqref{equ:lbe-2} over $i$ and use $\sum_{i=1}^5 g_i^{(eq)}=u$ to obtain
\begin{equation}\label{equ:lbe-3}
    \partial_t u + \frac{1}{\varepsilon} \nabla\cdot (\sum_{k=1}^5 \xi_k g_k) =  0.
\end{equation}
On the other hand, it follows from \eqref{equ:lbe-2} that
\begin{equation}\nonumber
    \begin{aligned}
        g_i ={}& g_i^{(eq)}-\varepsilon \tau \xi_i \cdot \nabla g_i - \varepsilon^2 \tau \partial_t g_i = \omega_i u + 3\varepsilon \omega_i \xi_i \cdot f(u) - \varepsilon \tau w_i \xi_i \cdot \nabla  u + O(\varepsilon^2).
    \end{aligned}
\end{equation}
Substituting above relation into \eqref{equ:lbe-3} gives
\begin{equation*}
 \begin{aligned}
    \partial_t u + \nabla \cdot f(u) = \nabla \cdot (\frac{1}{3}\tau \nabla u)
    + O(\varepsilon^2).
\end{aligned}
\end{equation*}
From this we see that the limit equation of $u$ for the lattice Boltzmann model \eqref{equ:lbe-2} is exactly the CDE \eqref{equ:lbe-1} if  $\tau=3D$ is adopted.

To verify the convergence criteria, we denote $h_i = g_i - w_i u$ and use \eqref{equ:lbe-2} and \eqref{equ:lbe-3} to obtain
\begin{equation}\label{equ:lbe-5}
\begin{aligned}
    &\partial_t u + \frac{1}{\varepsilon} \nabla\cdot (\sum_{k=1}^5 \xi_k h_k) =  0,\\
    &\partial_t h_i + \frac{1}{\varepsilon}  \nabla \cdot \brac{\xi_i h_i + \xi_i w_i u - w_i \sum_{k=1}^5 \xi_k h_k } = \frac{1}{\varepsilon^2} \frac{1}{\tau}(3\varepsilon \omega_i \xi_i \cdot f(u) - h_i), \quad i=1,\ldots, 5.
\end{aligned}  
\end{equation}
With $U=(u, w)^T$ and $w=(h_2, \ldots, h_5)^T$,  \eqref{equ:lbe-5} can be written in the form of \eqref{equ:2-1-pde} as
\begin{equation}\label{equ:key-lbe}
\begin{aligned}
    \partial_t U + \frac{1}{\varepsilon}\sum_{j=1}^2 A_j \partial_{x_j} U = \frac{1}{\varepsilon^2}Q(U;\varepsilon),
\end{aligned}    
\end{equation}
where 
\begin{equation*}
\begin{aligned}\label{equ:key-lbe-2}
    &A_j = \small \begin{pmatrix}
        0 & \xi_2^{(j)}-\xi_1^{(j)} & \xi_3^{(j)}-\xi_1^{(j)} & \xi_4^{(j)}-\xi_1^{(j)} & \xi_5^{(j)}-\xi_1^{(j)} \\[2mm]
        \omega_2 \xi_2^{(j)} & \xi_2^{(j)}-\omega_2(\xi_2^{(j)}-\xi_1^{(j}) & -\omega_2(\xi_3^{(j)}-\xi_1^{(j)}) & -\omega_2(\xi_4^{(j)}-\xi_1^{(j)}) & -\omega_2(\xi_5^{(j)}-\xi_1^{(j)}) \\[2mm]
        \omega_3 \xi_3^{(j)} & -\omega_3(\xi_2^{(j)}-\xi_1^{(j)}) & \xi_3^{(j)}-\omega_3(\xi_3^{(j)}-\xi_1^{(j)}) & -\omega_3(\xi_4^{(j)}-\xi_1^{(j)}) & -\omega_3(\xi_5^{(j)}-\xi_1^{(j)}) \\[2mm]
        \omega_4 \xi_4^{(j)} & -\omega_4(\xi_2^{(j)}-\xi_1^{(j)}) & -\omega_4(\xi_3^{(j)}-\xi_1^{(j)}) & \xi_4^{(j)}-\omega_4(\xi_4^{(j)}-\xi_1^{(j)}) & -\omega_4(\xi_5^{(j)}-\xi_1^{(j)}) \\[2mm]
        \omega_5 \xi_5^{(j)} & -\omega_5(\xi_2^{(j)}-\xi_1^{(j)}) & -\omega_5(\xi_3^{(j)}-\xi_1^{(j)}) & -\omega_5(\xi_4^{(j)}-\xi_1^{(j)}) & \xi_5^{(j)}-\omega_5(\xi_5^{(j)}-\xi_1^{(j)})\\[2mm]
    \end{pmatrix},\\[3mm] 
   & Q(U;\varepsilon)
    = (0, ~q(u, w;\varepsilon))^T
    :=\brac{0, ~\frac{1}{\tau}( 3\varepsilon \omega_2 \xi_2 \cdot f(u)  - h_2 ), \ldots, \frac{1}{\tau}( 3\varepsilon \omega_5 \xi_5 \cdot f(u)  - h_5) }^T.
\end{aligned}    
\end{equation*}
In the above equation $\xi_i^{(j)}$ is the $j$-th component of $\xi_i$. 
%
Now we verify Conditions \ref{cond:2-ssc-1}-\ref{cond:2-ssc-5} for the system \eqref{equ:key-lbe}.
Note that Conditions \ref{cond:2-ssc-1} and \ref{cond:2-ssc-5} naturally holds as $A_j^{11}=0$. On the other hand, it is easy to see that
\begin{equation}\nonumber
  \begin{aligned}
      q(u, w; 0) = 0 \iff w=(h_2,\dots, h_5)^T = 0
  \end{aligned}  
\end{equation}
and 
\begin{equation}\label{n316}
\partial_U Q(u, 0;0) = -\frac{1}{\tau}\diag\{0, ~1, \ldots, 1\},
\end{equation}
thus Condition \ref{cond:2-ssc-2} is true. 
Moreover, define 
\begin{equation}
\begin{aligned}\nonumber
    P = \begin{pmatrix}
        1 & 1 & 1 & 1 & 1\\
        -\omega_2 & 1-\omega_2 & -\omega_2 & -\omega_2 & -\omega_2 \\
        -\omega_3 & -\omega_3 & 1-\omega_3 & -\omega_3 & -\omega_3 \\
        -\omega_4 & -\omega_4 & -\omega_4 & 1-\omega_4 & -\omega_4 \\
        -\omega_5 & -\omega_5 & -\omega_5 & -\omega_5 & 1-\omega_5 \\
    \end{pmatrix},
\end{aligned}    
\end{equation}
which is invertible since $\det(P) = 1$.  It is direct to verify that
\begin{equation}\nonumber
    A_j P  
    = P \diag\{\xi_1^{(j)}, \xi_2^{(j)}, \ldots, \xi_5^{(j)}\},\quad j=1, ~2.
\end{equation}
Then the symmetrizer in Condition \ref{cond:2-ssc-3} can be chosen as 
\begin{equation}\nonumber
    A_0  = P^{-T}\diag\{\frac{1}{2}, ~1, \dots, ~1\} P^{-1} = \begin{pmatrix}
        \frac{1}{6} & 0\\[2mm]
        0 & \frac{1}{2}\mathbf{1}_4 \mathbf{1}_4^T + I_4
    \end{pmatrix},
\end{equation}
where $\mathbf{1}_4=(1,1,1,1)^T$. 
With the above two equations we see that $A_0 A_j, j=1,2$ are symmetric and Condition \ref{cond:2-ssc-3} is satisfied. Furthermore, it follows from \eqref{n316} and the definition of $A_0$ that 
\begin{equation}\nonumber
\begin{aligned}
A_0 \partial_U Q(u, 0;0)
=-\frac{1}{\tau}\begin{pmatrix}
   0 & 0  \\
   0 & \frac{1}{2}\mathbf{1}_4 \mathbf{1}_4^T + I_4
\end{pmatrix}\leq 0
\end{aligned}
\end{equation}
which implies Condition \ref{cond:2-ssc-4}. 
Thus we have verified all the conditions and justified the validity of the system \eqref{equ:key-lbe}, or equivalently the lattice Boltzmann model \eqref{equ:lbe-1}.

\subsection{Diffusive kinetic model for nonlinear parabolic systems}

In the last example, we consider a kinetic model in \cite{Aregba2001diffusion,aregba2004explicit} for the following multi-dimensional nonlinear parabolic system
\begin{equation}\label{equ:dk-1}
    \begin{aligned}
       \partial_t u + \sum_{j=1}^d \partial_{x_j} F_{j}(u) = \sum_{j=1}^d \partial_{x_j}^2 B(u),
    \end{aligned}
\end{equation}
where $u=(u_1, \dots, u_K)^T \in \Omega$ is function of $(x, t)=(x_1,\cdots,x_d, t)\in \mathbb{R}^d\times [0, +\infty)$, and $F_{j}= F_{j}(u)$ and $B=B(u)$ are $K$-vector valued functions of $u$.
System \eqref{equ:dk-1} is assumed to satisfy \cite{aregba2004explicit}
\begin{enumerate}[label=(B\arabic*)]
    \item \label{B1} For all $\xi \in \mathbb{R}^d$, $\sum_{j=1}^d \xi_j \partial_u F_{j}(u)$ has real eigenvalues and is diagonalizable.
    \item \label{B2} The  eigenvalues of $\partial_u B(u)$ are positive.
\end{enumerate}
Here we remark that Assumption \ref{B1} means the hyperbolicity of the convective part; Assumption \ref{B2} states that the system is strictly parabolic, which is required in our following analysis and is stronger than the original assumption of degenerate parabolic systems in \cite{aregba2004explicit}.

By extending the work in \cite{Aregba2001diffusion} for the scalar case, a diffusive kinetic relaxation model was proposed in \cite{aregba2004explicit} for system \eqref{equ:dk-1}:
\begin{equation}\label{equ:dk-2}
    \begin{aligned}
        &\partial_t f_l + \sum_{j=1}^d \lambda_{l j} \partial_{x_j} f_l =\frac{1}{\varepsilon} (M_l(u )-f_l ),\quad  1 \leq l \leq N, \\ 
        &\partial_t f_{N+m} + \sum_{j=1}^d \gamma^\varepsilon \sigma_{m j} \partial_{x_j} f_{N+m} =\frac{1}{\varepsilon} \brac{\frac{B (u )}{N^{\prime} \theta^2}-f_{N+m} },\quad  1 \leq m \leq N^{\prime},
    \end{aligned}
\end{equation}
where $u(x, t)=\sum_{l=1}^{N+N^{\prime}} f_l(x, t)\in \mathbb{R}^K$, each $f_l$ and $M_l$ take values in $\mathbb{R}^K, \varepsilon>0$ is a small parameter, $\lambda_{l j}$ are some fixed constants, $\gamma^\varepsilon=\mu+\frac{\theta \sqrt{N^{\prime}}}{\sqrt{\varepsilon}}$, $ \mu \geq 0, \theta>0, N^{\prime} \geq$ $d+1$, and $\sigma_{mj}$ are constants satisfying
\begin{equation}\nonumber
    \sum_{m=1}^{N^{\prime}} \sigma_{mj}=0, \quad
\sum_{m=1}^{N^{\prime}} \sigma_{mj} \sigma_{mi} = \delta_{ij}, \quad \forall i,j=1,2,\ldots,d.
\end{equation}
The function $M_l$ is called a local Maxwellian function and satisfies
\begin{equation}\label{equ:dk-3}
    \sum_{l=1}^N M_l(u) = u- \frac{B(u)}{\theta^2}, \quad \sum_{l=1}^N \lambda_{lj}M_l(u) = F_j(u), \quad j=1,\dots, d.
\end{equation}
Following \cite{aregba2004explicit}, the diffusive kinetic model is assumed to satisfy 
\begin{enumerate}[label=(C\arabic*)]
    \item \label{C1} The function $M_l(u)$ is a strictly monotone Maxwellian function (SMFF), \ie, the eigenvalues of $\partial_u M_{l}(u)$ are positive for all $u \in \Omega$.
    \item \label{C2} There exists a basis of common eigenvectors for the matrices $\sum_{j=1}^d \xi_j \partial_u F_{j}(u), \partial_u M_{l}(u)(l=1, \ldots, N)$ and $\partial_u B(u)$ for any $\xi \in \mathbb{R}^d$ with $|\xi|=1$ and for all $u \in \Omega$.
\end{enumerate}
Denote by $H$ the matrix consisting of the basis of common eigenvectors above. Then we deduce from Assumptions \ref{B2} and \ref{C1} that $\partial_u M_{l}$ and $\frac{\partial_u B(u)}{N'\theta^2}$ can be diagonalized as
\begin{equation}\label{equ:dk-3-2}
    \partial_u M_{l} = H^{-1} \Lambda_l H,\quad l=1,\dots, N, \qquad \frac{\partial_u B(u)}{N'\theta^2} = H^{-1} \Lambda_B H,
\end{equation}
where $\Lambda_l$ and $\Lambda_B$ are positive-definite diagonal matrices.

To derive the limit equation of \eqref{equ:dk-2}, we sum up \eqref{equ:dk-2} over the index $l$ from $1$ to $N+N^\prime$ and utilize relations \eqref{equ:dk-3} to obtain
\begin{equation}\nonumber
    \partial_t u + \sum_{j=1}^d\brac{ \sum_{l=1}^N \lambda_{lj} \partial_{x_j} f_l + \sum_{m=1}^{N^\prime} \gamma^\varepsilon \sigma_{mj}\partial_{x_j}f_{N+m}}  = 0.
\end{equation}
Denoting $g_l = f_l - M_l(u),l=1,\dots, N$ and $g_{N+m} = f_{N+m} - \frac{B (u)}{N^{\prime} \theta^2}, m=1,\dots, N'$, we have 
\begin{equation}\nonumber
    \begin{aligned}
        &\partial_t g_l + \sum_{j=1}^d \brac{ \sum_{i=1}^N (I_K -\partial_u M_{i}(u)) \lambda_{i j}\partial_{x_j} f_i -   \partial_u M_{l}(u)\sum_{i=1}^{N^\prime} \gamma^\varepsilon \sigma_{ij}\partial_{x_j}f_{N+i} } \\
        & \hspace{0.5cm}= - \frac{1}{\varepsilon} g_l, \quad 1 \leq l \leq N, \\ 
        &\partial_t g_{N+m} + \sum_{j=1}^d \brac{\brac{I_K-\frac{\partial_u B(u)}{N^{\prime} \theta^2}} \sum_{i=1}^{N^\prime} \gamma^\varepsilon \sigma_{ij} \partial_{x_j}f_{N+i} -\frac{\partial_u B(u)}{N^{\prime} \theta^2} \sum_{i=1}^N \lambda_{i j}\partial_{x_j} f_i} \\
        &\hspace{0.5cm}= -\frac{1}{\varepsilon} g_{N+m}, \quad 1 \leq m \leq N^{\prime}.
    \end{aligned}  
\end{equation}
For convenience, we use $ \varepsilon^2$ to replace $\varepsilon$ in what follows. Then $\gamma^\varepsilon=\mu+\frac{\theta \sqrt{N^{\prime}}}{\varepsilon} $ and the equations of $u $ and $g_l, l=2,\dots, N+N'$ can be reformulated as 
\begin{equation}
\begin{aligned}\label{equ:dk-6}
    \partial_t u &+ \frac{1}{\varepsilon}\sum_{j=1}^d\brac{\varepsilon \partial_u F_{j}(u )\partial_{x_j}u + \sum_{l=2}^N \varepsilon (\lambda_{lj}-\lambda_{1j})\partial_{x_j}g_l + \sum_{i=1}^{N^\prime} (\varepsilon \mu \sigma_{ij}+\theta\sqrt{N'} \sigma_{ij} - \varepsilon\lambda_{1j})\partial_{x_j}g_{N+i} }\\
    ={}&  0,\\[2mm]
    \partial_t g_l &+ \frac{1}{\varepsilon} \sum_{j=1}^d \brac{\varepsilon (\lambda_{lj}- \partial_u F_{j}(u )) \partial_u M_{l}(u)\partial_{x_j} u +  \varepsilon \lambda_{lj} \partial_{x_j} g_l - \sum_{i=2}^N \varepsilon \partial_u M_{l}(u )(\lambda_{ij}-\lambda_{1j})\partial_{x_j}g_i } \\
    &-\frac{1}{\varepsilon} \sum_{j=1}^d \partial_u M_{l}(u)\sum_{i=1}^{N^\prime} (\varepsilon \mu \sigma_{ij}+\theta\sqrt{N'} \sigma_{ij} - \varepsilon\lambda_{1j}) \partial_{x_j}g_{N+i} \\
    ={}& - \frac{1}{\varepsilon^2} g_l, \quad l=2,\dots, N, \\[2mm] 
    \partial_t g_{N+m} & + \frac{1}{\varepsilon}\frac{ \partial_u B(u)}{N^{\prime} \theta^2}\sum_{j=1}^d \brac{ \brac{ (\varepsilon \mu + \theta \sqrt{N'})\sigma_{mj}- \varepsilon \partial_u F_{j}(u) } \partial_{x_j}u - \sum_{i=2}^N \varepsilon(\lambda_{ij}-\lambda_{1j}) \partial_{x_j}g_i } \\
    &+ \frac{1}{\varepsilon}\sum_{j=1}^d \sigma_{mj} (\varepsilon\mu + \theta\sqrt{N'})\partial_{x_j}g_{N+m} - \frac{1}{\varepsilon}\frac{\partial_u  B(u)}{N^{\prime} \theta^2}\sum_{j=1}^d \sum_{i=1}^{N^\prime} (\varepsilon \mu \sigma_{ij}+\theta\sqrt{N'} \sigma_{ij} - \varepsilon\lambda_{1j}) \partial_{x_j}g_{N+i} \\
    ={}& -\frac{1}{\varepsilon^2} g_{N+m}, \quad m=1,\dots, N'. 
\end{aligned}
\end{equation}
Here we have used relations in \eqref{equ:dk-3}, $\sum_{l=1}^{N+N^{\prime}} g_l = 0$, and  
\begin{equation}
    \begin{aligned}\nonumber
        &\sum_{l=1}^N \lambda_{lj} f_l = \sum_{l=1}^N \lambda_{lj}g_l + \sum_{l=1}^N \lambda_{lj}M_l(u) = \sum_{l=2}^N (\lambda_{lj}-\lambda_{1j})g_l - \lambda_{1j}\sum_{m=1}^{N'}g_{N+m} + F_j(u), \\
        &\sum_{m=1}^{N^\prime} \gamma^\varepsilon \sigma_{mj} f_{N+m} =\sum_{m=1}^{N^\prime} \gamma^\varepsilon \sigma_{mj} g_{N+m} + \sum_{m=1}^{N^\prime} \gamma^\varepsilon \sigma_{mj} \frac{B(u )}{N'\theta^2}=\sum_{m=1}^{N^\prime} \gamma^\varepsilon \sigma_{mj} g_{N+m}.
    \end{aligned}
\end{equation}
Moreover, it follows from \eqref{equ:dk-6} that
\begin{equation}
\begin{aligned}\nonumber
    g_l ={}& \varepsilon \sum_{j=1}^d \partial_u M_{l}(u ) \theta \sqrt{N'}\sum_{i=1}^{N'} \sigma_{ij}\partial_{x_j} g_{N+i} + O(\varepsilon^2),\\
    g_{N+m} ={}& \varepsilon \sum_{j=1}^d \brac{\frac{\partial_u B(u)}{\sqrt{N'}\theta}\sigma_{mj}\partial_{x_j}u -\theta\sqrt{N'} \sigma_{mj}\partial_{x_j} g_{N+m} + \frac{\partial_u B(u)}{\sqrt{N'}\theta}\sum_{i=1}^{N'} \sigma_{ij}\partial_{x_j} g_{N+i}} + O(\varepsilon^2),
\end{aligned}    
\end{equation}
which further yields
\begin{equation}
\begin{aligned}\nonumber
    g_l  ={}& O(\varepsilon^2), \quad l = 2, \dots, N,\\
    g_{N+m}  ={}& \varepsilon \sum_{j=1}^d \frac{\partial_u B(u)}{\sqrt{N'}\theta}\sigma_{mj}\partial_{x_j}u + O(\varepsilon^2), \quad m=1,\dots, N'.
\end{aligned}    
\end{equation}
Substituting these into the equation of $u $ in \eqref{equ:dk-6} gives  
\begin{equation}
\begin{aligned}\nonumber
    &\partial_t u + \sum_{j=1}^d \partial_u F_{j}(u)\partial_{x_j}u +\sum_{j,k=1}^d \sum_{m=1}^{N'} \partial_{x_j}\partial_u B(u) \sigma_{mj}\sigma_{mk}\partial_{x_k}u =0,
\end{aligned}    
\end{equation}
which is the parabolic system \eqref{equ:dk-1} by noticing that $\sum_{m=1}^{N^{\prime}} \sigma_{mj} \sigma_{mk} = \delta_{jk}$.

To verify the convergence criteria, we 
define $K(N+N') \times K(N+N')$ matrices 
\begin{equation}
\begin{aligned}\nonumber
    \tilde{A}_j  = \diag\{\varepsilon\lambda_{1j}I_K, ~ \varepsilon\lambda_{2j}I_K, \dots, \varepsilon\lambda_{Nj}I_K, ~ \sigma_{1j} (\varepsilon\mu + \theta\sqrt{N'})I_K, \dots, \sigma_{N'j} (\varepsilon\mu + \theta\sqrt{N'})I_K\}
\end{aligned}    
\end{equation}
with $j=1,2,...,d$
and 
\begin{equation}
\begin{aligned}\nonumber
    P = \begin{pmatrix}
        I_K & I_K & \dots & I_K & I_K & \dots & I_K\\
        -\partial_u M_{2}(u) & I_K-\partial_u M_{2}(u) & \dots  & -\partial_u M_{2}(u) & -\partial_u M_{2}(u) & \dots & -\partial_u M_{u}(u) \\
        \vdots & \vdots & \ddots  & \vdots  & \vdots & \ddots & \vdots\\
        -\partial_u M_{N}(u) & -\partial_u M_{N}(u) & \dots  & I_K-\partial_u M_{N}(u) & -\partial_u M_{N}(u) & \dots & -\partial_u M_{N}(u) \\[2mm]
       -\frac{\partial_u B(u)}{N^{\prime} \theta^2} &  -\frac{\partial_u B(u)}{N^{\prime} \theta^2} & \dots &  -\frac{\partial_u B(u)}{N^{\prime} \theta^2} &  I_K-\frac{\partial_u B(u)}{N^{\prime} \theta^2} & \dots &  -\frac{\partial_u B(u)}{N^{\prime} \theta^2}\\
       \vdots & \vdots & \ddots  & \vdots  & \vdots & \ddots & \vdots\\
       -\frac{\partial_u B(u)}{N^{\prime} \theta^2} &  -\frac{\partial_u B(u)}{N^{\prime} \theta^2} & \dots &  -\frac{\partial_u B(u)}{N^{\prime} \theta^2} &  -\frac{\partial_u B(u)}{N^{\prime} \theta^2} & \dots & I_K -\frac{\partial_u B(u)}{N^{\prime} \theta^2}
    \end{pmatrix}.
\end{aligned}    
\end{equation}
%
It is direct to compute that $\det(P)=1$ and thus $P^{-1}$ exists.
Denote $U=(u , w )\in \mathbb{R}^{K(N+N')}$ with $w =(g_2, \dots, ~g_N, ~g_{N+1}, \dots, ~g_{N+N'} )^T\in \mathbb{R}^{K(N-1+N')}$.
The system \eqref{equ:dk-6} can be written as 
\begin{equation}\label{equ:dk-8}
\begin{split}
&\partial_t U + \frac{1}{\varepsilon} \sum_{j=1}^d A_{j}(U;\varepsilon) \partial_{x_j} U = \frac{1}{\varepsilon^2} Q(U), \\
& A_{j}(U;\varepsilon) =P \tilde{A}_j P^{-1}, \quad j=1,\ldots, d, \quad Q(U) =- (0, ~w)^T.
\end{split}
\end{equation}

Next we show that system \eqref{equ:dk-8} admits Conditions \ref{cond:2-ssc-1}-\ref{cond:2-ssc-5}. 
First, Conditions \ref{cond:2-ssc-1} and \ref{cond:2-ssc-2} are obviously true since 
\begin{equation}\nonumber
  q(u, w; 0) = 0 \iff w = 0
\end{equation}
and
\begin{equation}\label{equ:dk-9}
\partial_U Q(u, 0;0) = - \diag\{0, 1, \ldots, 1\}.  
\end{equation}
On the other hand, it follows from the first equation in \eqref{equ:dk-6} that $A_j^{11}(u,w;\varepsilon) = \varepsilon F_j'(u )$, and thus Condition \ref{cond:2-ssc-5} is satisfied.
To verify the other two conditions, we define the symmetrizer as 
\begin{equation}\nonumber
   A_0 = P^{-T} \tilde{A}_0 P^{-1}, 
\end{equation}
where
\begin{equation}
\begin{aligned}\nonumber
    \tilde{A}_0 = \diag\{ H^{T} \Lambda_1^{-1} H, ~H^{T} \Lambda_2^{-1} H, \dots, H^{T} \Lambda_N^{-1} H , ~H^{T} \Lambda_B^{-1} H, \dots, H^{T} \Lambda_B^{-1} H \}~\dot{=}~\diag\{\tilde{A}_0^{11},\tilde{A}_0^{22} \}
\end{aligned}
\end{equation}
with $\tilde{A}_0^{11} = H^{T} \Lambda_1^{-1} H $. 
Note that $A_0$ is positive definite according to the definitions of $\Lambda_l, l=1,2,\ldots,N$ and $\Lambda_B$ defined in \eqref{equ:dk-3-2} and Assumptions \ref{B2} and \ref{C1}. Then we have $A_0 A_j = P^{-T} \tilde{A}_0 \tilde{A}_j P^{-1}$, which is symmetric since so is $\tilde{A}_0 \tilde{A}_j$. Thus Condition \ref{cond:2-ssc-3} is satisfied.
Furthermore, it follows from \eqref{equ:dk-9} and the expressions of $A_0$ and $P^{-1}$ that 
\begin{equation}\nonumber
    \begin{aligned}
       A_0 \partial_U Q = {}& -P^{-T} \tilde{A}_0 P^{-1} \diag\{0, 1, \ldots, 1\}
       ={}  -\begin{pmatrix}
           O_{K\times K} & O_{K\times K(N+N'-1)} \\
           O_{K(N+N'-1) \times K} & S
       \end{pmatrix}
    \end{aligned}
\end{equation}
with
\begin{equation}\nonumber
    S= \mathbf{1}_{K(N+N'-1)} \tilde{A}_0^{11} \mathbf{1}_{K(N+N'-1)}^T 
    + 
    \tilde{A}_0^{22},\quad \mathbf{1}_{K(N+N'-1)} = (I_K, \dots, I_K)^T\in \mathbb{R}^{K(N+N'-1)\times K}.
\end{equation}
%
Note that $\mathbf{1}_{K(N+N'-1)} \tilde{A}_0^{11} \mathbf{1}_{K(N+N'-1)}^T$ is non-negative definite and $\tilde{A}_0^{22}$ is positive definite, thus $S$ is positive definite too. Therefore, Condition \ref{cond:2-ssc-4} holds and the validity of the diffusive kinetic model \eqref{equ:dk-2} for the parabolic system \eqref{equ:dk-1} is justified.

\section{A new relaxation model}
\label{sec4}

Note that the systems approximated by relaxation models in the previous section are all assumed to be strictly parabolic. In this section, we  construct a relaxation model for the following general hyperbolic-parabolic system:
\begin{equation}\label{11}
\partial_t u + \sum_{j=1}^d a_j\partial_{x_j} u
= \sum_{j=1}^d \partial_{x_j}( \sum_{k=1}^d D_{jk}\partial_{x_k} u),
\end{equation}
where $u=(u_1, u_2)^T\in \mathbb R^m$, $u_1 \in \mathbb R^{m-s}, u_2 \in \mathbb R^{s}$, and
\begin{equation}\nonumber
\begin{aligned}
& a_j = a_j(u)=
\begin{pmatrix}
   a_j^{11} & a_j^{12}\\[2mm]
   a_j^{21} & a_j^{22} 
\end{pmatrix}, 
\quad a_j^{11}=a_j^{11}(u) \in \mathbb R^{(m-s) \times (m-s)}, \quad a_j^{22}=a_j^{22}(u) \in \mathbb R^{s \times s}, \\
& D_{jk} = D_{jk}(u)=
\begin{pmatrix}
O_{(m-s)\times (m-s)} & O_{(m-s)\times s} \\[2mm]
D_{jk}^{21} & D_{jk}^{22}
\end{pmatrix}, D_{jk}^{21}=D_{jk}^{21}(u) \in \mathbb R^{s \times (m-s)}, \quad  D_{jk}^{22}=D_{jk}^{22}(u) \in \mathbb R^{s \times s}.
\end{aligned}
\end{equation}
We assume system \eqref{11} is endowed with a symmetrizer $a_0 = a_0(u)$ such that
\begin{enumerate}[label=(\Roman*)]

\item \label{R1} $a_0 a_j$ is symmetric for all $j=1,2,...,d$.

\item \label{R2} The $md \times md$ matrix
\begin{equation}\label{z3}
\begin{pmatrix}
    a_0D_{11} & a_0D_{12} & ... & a_0D_{1d}\\
    a_0D_{21} & a_0D_{22} & ... & a_0D_{2d}\\
    \vdots & \vdots & ... & \vdots\\
    a_0D_{d1} & a_0D_{d2} & ... & a_0D_{dd}
\end{pmatrix}
\end{equation}
is symmetric positive semidefinite.
\end{enumerate}
Additionally, we further assume
\begin{enumerate}[label=(III)]
\item \label{R3} The $sd \times sd$ matrix
\begin{equation}\label{14}
H=
\begin{pmatrix}
    D_{11}^{22} & D_{12}^{22} & ... & D_{1d}^{22}\\
    D_{21}^{22} & D_{22}^{22} & ... & D_{2d}^{22}\\
    \vdots & \vdots & ... & \vdots\\
    D_{d1}^{22} & D_{d2}^{22} & ... & D_{dd}^{22}
\end{pmatrix}
\end{equation}
is invertible.
\end{enumerate}

Here Assumptions \ref{R1}-\ref{R3} guarantee the existence of solutions to hyperbolic-parabolic systems of the form \eqref{11} and are admitted by the compressible Navier-Stokes equations \cite{kawashima1984systems,kawashima1988}. Particularly, Assumption \ref{R3} means that the parabolic part of the system is strictly parabolic, \ie, $\diag\{a_0,a_0,...,a_0\} H$ is positive definite according to Assumptions \ref{R2} and \ref{R3}.

We construct the following relaxation system for \eqref{11}:
\begin{subequations}\label{15}
\begin{align}
& \partial_t u_1+ \sum_{j=1}^d a_j^{11}\partial_{x_j} u_1+ \sum_{j=1}^d a_j^{12}\partial_{x_j} u_2=0, \\
& \partial_t u_2+ \sum_{j=1}^d a_j^{21}\partial_{x_j} u_1+ \sum_{j=1}^d a_j^{22}\partial_{x_j} u_2
  + \frac{1}{\varepsilon} \sum_{j=1}^d \partial_{x_j} w_j  = 0, \\
& \partial_t w_j  + \frac{1}{\varepsilon}(\sum_{k=1}^d D_{jk}^{21}\partial_{x_k} u_1+ \sum_{k=1}^d D_{jk}^{22}\partial_{x_k} u_2) = - \frac{1}{\varepsilon^2} w_j ,\quad j=1,2,...,d,
\end{align}
\end{subequations}
where $w_j \in \mathbb R^{s}$ is an approximation of $-\varepsilon (\sum_{k=1}^d D_{jk}^{21}\partial_{x_k} u_1 + \sum_{k=1}^d D_{jk}^{22}\partial_{x_k} u_2)$. The system can be rewritten as
\begin{equation}\label{16}
\partial_t U
+ \sum_{j=1}^d \bar A_j\partial_{x_j} U + \frac{1}{\varepsilon}\sum_{j=1}^d \hat A_j\partial_{x_j} U
= \frac{1}{\varepsilon^2}Q,
\end{equation}
where $U = (u_1, u_2, w_1, w_2, \ldots, w_d)^T \in \mathbb R^{m+sd}$,  $Q=Q(U)=\diag\{0,-w_1,-w_2, \ldots,-w_d\}$ and
\begin{equation}
\begin{split}
& \bar A_j = \bar A_j(U)=
\left(
\begin{array}{lll}
a_j^{11} & a_j^{12} & O_{(m-s)\times ds}\\[2mm]
a_j^{21} & a_j^{22} & O_{s \times ds}\\[2mm]
O_{ds \times (m-s)} & O_{ds \times s} & O_{ds \times ds}
\end{array}
\right), \\[2mm]
& \hat A_j = \hat A_j(U)=
\left(
\begin{array}{lllll}
O_{(m-s)\times (m-s)} & O_{(m-s)\times s} & O_{(m-s)\times s} & ... & O_{(m-s)\times s}\\[2mm]
O_{s \times (m-s)} & O_{s \times s} & \delta_{j1}I_{s} & ... & \delta_{jd}I_{s}\\[2mm]
D_{1j}^{21} & D_{1j}^{22} & O_{s\times s} & ... & O_{s\times s}\\
\vdots & \vdots & \vdots &  ... & \vdots \\
D_{dj}^{21} & D_{dj}^{22} & O_{s\times s} & ... & O_{s\times s}
\end{array}
\right).
\end{split}
\end{equation}

Let
\begin{equation*}
a_0 =
\left(
\begin{array}{cc}
a_0^{11} & a_0^{12} \\[2mm]
a_{0}^{21} & a_{0}^{22}
\end{array}
\right).
\end{equation*}
According to Assumption \ref{R3}, $H$ in \eqref{14} is invertible and thus we define
\begin{equation}\label{z8}
B=\diag\{a_{0}^{22}, a_{0}^{22}, \ldots, a_{0}^{22}\} H^{-1}
\end{equation}
and
\begin{equation}\nonumber
A_0 = \diag\{a_{0}, ~B\}.
\end{equation}
Then on the system \eqref{16}, we have

\begin{theorem}\label{thm11}
Under Assumptions \ref{R1}-\ref{R3} for the hyperbolic parabolic system \eqref{11}, the relaxation model \eqref{16} is symmetrizable hyperbolic in the sense that
\begin{enumerate}[label=(\alph*)]
    \item \label{a} The matrices $B$ and $A_0$ are symmetric positive definite.
    \item \label{b} The matrices $A_0 \bar A_j$ and $A_0 \hat A_j$ are symmetric for all $j=1,2,...,d$.
    \item \label{c} The matrix $A_0 Q_{U} = \diag\{0,  -B\}$ is symmetric negative semidefinite.
\end{enumerate}
\end{theorem}

\begin{proof}
We first prove \ref{a}.  Since $H^T B H = H^T \diag\{a_{0}^{22}, a_{0}^{22}, \ldots, a_{0}^{22}\}$ and $a_{0}^{22}$ is symmetric positive definite, we turn to show the symmetry and positive definiteness of $\diag\{a_{0}^{22}, a_{0}^{22}, \ldots, a_{0}^{22}\} H$. Noticing that $\diag\{a_{0}^{22}, a_{0}^{22}, \ldots, a_{0}^{22}\} H$ is a $md \times md$ sub-matrix of the matrix in \eqref{z3}, thus it is symmetric positive semidefinite. On the other hand, according to Assumption \ref{R3}, \ie, $H$ is invertible, the matrix $\diag\{a_{0}^{22}, a_{0}^{22}, \ldots, a_{0}^{22}\} H$ is invertible and thus is symmetric positive definite. Therefore, $B$ is symmetric positive definite and that of $A_0$ immediately follows.
Additionally, conclusion \ref{c} also follows from this.

Next we prove \ref{b}. The symmetry of $A_0 \bar A_j$ is obvious thanks to the symmetry of $a_0 a_j$ in Assumption \ref{R1}.
To show the symmetry of $A_0 \hat A_j$, we divide the matrix $B$ in the same manner as $H$ in \eqref{14} and denote the $(j,k)$-th sub-matrix as $B_{jk} \in \mathbb R^{s \times s}$. Then we directly compute
\begin{equation*}
\begin{split}
&A_0 \hat A_j = \diag\{A_{0}, B\} \hat A_j\\[3mm]
={}&
\setlength{\arraycolsep}{2pt}
\left(
\begin{array}{lllll}
a_0^{11} & a_0^{12} & O_{(m-s)\times s} & ... & O_{(m-s)\times s}\\[2mm]
a_{0}^{21} & a_{0}^{22} & O_{s\times s} & ... & O_{s\times s}\\[2mm]
O_{s\times (m-s)} & O_{s\times s} & B_{11} & ... & B_{1d}\\
\vdots & \vdots & \vdots &  ... & \vdots \\
O_{s\times (m-s)} & O_{s\times s} & B_{d1} & ... & B_{dd}
\end{array}
\right)
\left(
\begin{array}{lllll}
O_{(m-s)\times (m-s)} & O_{(m-s)\times s} & O_{(m-s)\times s} & ... & O_{(m-s)\times s}\\[2mm]
O_{s \times (m-s)} & O_{s \times s} & \delta_{j1}I_{s} & ... & \delta_{jd}I_{s}\\[2mm]
D_{1j}^{21} & D_{1j}^{22} & O_{s\times s} & ... & O_{s\times s}\\
\vdots & \vdots & \vdots &  ... & \vdots \\[2mm]
D_{dj}^{21} & D_{dj}^{22} & O_{s\times s} & ... & O_{s\times s}
\end{array}
\right)\\[3mm]
={}&
\left(
\begin{array}{lllll}
O_{(m-s)\times (m-s)} & O_{(m-s)\times s} & \delta_{j1}a_0^{12} & ... & \delta_{jd}a_0^{12}\\[2mm]
O_{s \times (m-s)} & O_{s \times s} & \delta_{j1}a_0^{22} & ... & \delta_{jd}a_0^{22}\\[2mm]
\sum_{k=1}^d B_{1k} D_{kj}^{21} & \sum_{k=1}^d B_{1k} D_{kj}^{22} & O_{s\times s} & ... & O_{s\times s}\\
\vdots & \vdots & \vdots &  ... & \vdots \\[2mm]
\sum_{k=1}^d B_{dk} D_{kj}^{21} & \sum_{k=1}^d B_{dk} D_{kj}^{22} & O_{s\times s} & ... & O_{s\times s}
\end{array}
\right).
\end{split}
\end{equation*}
The above matrix is symmetric if and only if
\begin{equation*}
B
\left(
\begin{array}{c}
D_{1j}^{21}\\
\vdots \\
D_{dj}^{21}
\end{array}
\right)
=
\left(
\begin{array}{c}
\delta_{j1}(a_0^{12})^T\\
\vdots \\
\delta_{jd}(a_0^{12})^T
\end{array}
\right), \quad
B
\left(
\begin{array}{c}
D_{1j}^{22}\\
\vdots \\
D_{dj}^{22}
\end{array}
\right)
=
\left(
\begin{array}{c}
\delta_{j1}(a_0^{22})^T\\
\vdots \\
\delta_{jd}(a_0^{22})^T
\end{array}
\right).
\end{equation*}
The above relations for all $j=1,2,...,d$ yield
\begin{subequations}
\begin{align}
& B
\left(
\begin{array}{ccc}
D_{11}^{21} & ... & D_{1d}^{21}\\
\vdots & \vdots & \vdots\\
D_{d1}^{21} & ... & D_{dd}^{21}
\end{array}
\right)
=
\diag\{(a_0^{12})^T, \ldots, (a_0^{12})^T\}, \label{z10a} \\[3mm]
& B
\left(
\begin{array}{ccc}
D_{11}^{22} & ... & D_{1d}^{22}\\
\vdots & \vdots & \vdots\\
D_{d1}^{22} & ... & D_{dd}^{22}
\end{array}
\right)
=
\diag\{(a_0^{22})^T, \ldots,  (a_0^{22})^T\}. \label{z10b}
\end{align}
\end{subequations}
Equation \eqref{z10b} obviously holds true according to the definition of $B$ in \eqref{z8}.
To show \eqref{z10a}, we deduce from \eqref{z10b} that
\begin{equation}\label{z11}
\begin{split}
&
\left(
\begin{array}{ccc}
D_{11}^{21} & \ldots & D_{1d}^{21}\\
\vdots & \vdots & \vdots\\
D_{d1}^{21} & \ldots & D_{dd}^{21}
\end{array}
\right)
^T
B
\left(
\begin{array}{ccc}
D_{11}^{22} & ... & D_{1d}^{22}\\
\vdots & \vdots & \vdots\\
D_{d1}^{22} & ... & D_{dd}^{22}
\end{array}
\right)
=
\left(
\begin{array}{ccc}
D_{11}^{21} & ... & D_{1d}^{21}\\
\vdots & \vdots & \vdots\\
D_{d1}^{21} & ... & D_{dd}^{21}
\end{array}
\right)
^T
\diag\{(a_0^{22})^T, \ldots, (a_0^{22})^T\}\\[3mm]
={}&
\left(
\begin{array}{ccc}
(a_0^{22}D_{11}^{21})^T & ... & (a_0^{22}D_{d1}^{21})^T\\
\vdots & \vdots & \vdots\\
(a_0^{22}D_{1d}^{21})^T & ... & (a_0^{22}D_{dd}^{21})^T
\end{array}
\right).
\end{split}
\end{equation}
On the other hand, according to Assumption \ref{R2}, \ie, the matrix in \eqref{z3} is symmetric, we have
\begin{equation*}
(a_0 D_{ij})^T = a_0D_{ji}.
\end{equation*}
That is
\begin{equation*}
\begin{split}
& \left(
\begin{array}{cc}
O_{(m-s)\times (m-s)} & O_{(m-s)\times s} \\[2mm]
D_{ij}^{21} & D_{ij}^{22}
\end{array}
\right)^T
\left(
\begin{array}{cc}
a_0^{11} & a_0^{12} \\[2mm]
a_{0}^{21} & a_{0}^{22}
\end{array}
\right)^T
=
\left(
\begin{array}{cc}
(D_{ij}^{21})^T (a_0^{12})^T & (D_{ij}^{21})^T (a_0^{22})^T \\[2mm]
(D_{ij}^{22})^T (a_0^{12})^T & (D_{ij}^{22})^T (a_0^{22})^T
\end{array}
\right)\\[2mm]
=
&
\left(
\begin{array}{cc}
a_0^{11} & a_0^{12} \\[2mm]
a_{0}^{21} & a_{0}^{22}
\end{array}
\right)
\left(
\begin{array}{cc}
O_{(m-s)\times (m-s)} & O_{(m-s)\times s} \\[2mm]
D_{ji}^{21} & D_{ji}^{22}
\end{array}
\right)
=
\left(
\begin{array}{cc}
a_0^{12}D_{ji}^{21} & a_0^{12}D_{ji}^{22}\\[2mm]
a_0^{22}D_{ji}^{21} & a_0^{22}D_{ji}^{22}
\end{array}
\right),
\end{split}
\end{equation*}
from which we have
\begin{equation*}
(D_{ij}^{21})^T (a_0^{22})^T = ( a_0^{22} D_{ij}^{21} )^T = a_0^{12}D_{ji}^{22}.
\end{equation*}
Substituting the above equations for all $i,j=1,2,...,d$ into \eqref{z11} gives
\begin{equation*}
\begin{split}
&\left(
\begin{array}{ccc}
D_{11}^{21} & ... & D_{1d}^{21}\\
\vdots & \vdots & \vdots\\
D_{d1}^{21} & ... & D_{dd}^{21}
\end{array}
\right)^T
B
\left(
\begin{array}{ccc}
D_{11}^{22} & ... & D_{1d}^{22}\\
\vdots & \vdots & \vdots\\
D_{d1}^{22} & ... & D_{dd}^{22}
\end{array}
\right) 
={}
\left(
\begin{array}{ccc}
(a_0^{22}D_{11}^{21})^T & ... & (a_0^{22}D_{d1}^{21})^T\\
\vdots & \vdots & \vdots\\
(a_0^{22}D_{1d}^{21})^T & ... & (a_0^{22}D_{dd}^{21})^T
\end{array}
\right)\\[3mm]
={}&
\left(
\begin{array}{ccc}
a_0^{12}D_{11}^{22} & ... & a_0^{12}D_{1d}^{22}\\
\vdots & \vdots & \vdots\\
a_0^{12}D_{d1}^{22} & ... & a_0^{12}D_{dd}^{22}
\end{array}
\right)
={}
\diag\{a_0^{12}, ..., a_0^{12}\}
\left(
\begin{array}{ccc}
D_{11}^{22} & ... & D_{1d}^{22}\\
\vdots & \vdots & \vdots\\
D_{d1}^{22} & ... & D_{dd}^{22}
\end{array}
\right).
\end{split}
\end{equation*}
Thanks to the invertibility of $H$ in Assumption \ref{R3}, we deduce from the above equation that
\begin{equation*}
\begin{split}
&
\left(
\begin{array}{ccc}
D_{11}^{21} & ... & D_{1d}^{21}\\
\vdots & \vdots & \vdots\\
D_{d1}^{21} & ... & D_{dd}^{21}
\end{array}
\right)
^T
B
=
\diag\{a_0^{12}, \ldots, a_0^{12}\}.
\end{split}
\end{equation*}
Then \eqref{z10a} follows from the transpose of the above equation. This completes the proof.

\end{proof}

\begin{theorem}\label{thm42}
Under Assumptions \ref{R1}-\ref{R3} for the hyperbolic-parabolic system \eqref{11}, the relaxation model \eqref{16} satisfies the convergence criteria \ref{cond:2-ssc-1}-\ref{cond:2-ssc-5}.
\end{theorem}

\begin{proof}
Note that Conditions \ref{cond:2-ssc-1} and \ref{cond:2-ssc-2} are naturally true.
Conditions \ref{cond:2-ssc-3} and \ref{cond:2-ssc-4} follow from the conclusions \ref{b} and \ref{c} in Theorem \ref{thm11}, respectively.
On the other hand, the relaxation model \eqref{16} can be written in the form of  \eqref{equ:2-1-pde} with
\begin{equation}
\begin{aligned}\nonumber
    A_j(U;\varepsilon)
     =\varepsilon\bar A_j + \hat A_j 
    = \begin{pmatrix}
        \varepsilon a_j^{11} & \varepsilon a_j^{12} & O_{(m-s)\times s} & \ldots & O_{(m-s)\times s}\\[2mm]
        \varepsilon a_j^{21} & \varepsilon a_j^{22} & \delta_{j1}I_{s} & \ldots & \delta_{jd}I_{s}\\[2mm]
        D_{1j}^{21} & D_{1j}^{22} & O_{s\times s} & \ldots & O_{s\times s}\\
        \vdots & \vdots & \vdots & \ldots & \vdots \\[2mm]
        D_{dj}^{21} & D_{dj}^{22} & O_{s\times s} & \ldots & O_{s\times s}
    \end{pmatrix}, \quad j=1,\dots, d.
\end{aligned}
\end{equation}
With the same partition as that in \eqref{equ:2-1-pde}, we have 
\begin{equation}\nonumber
  A_j^{11}(U;\varepsilon) = \begin{pmatrix}
      \varepsilon a_i^{11} & \varepsilon a_i^{12}\\[2mm]
      \varepsilon a_i^{21} & \varepsilon a_i^{22}
  \end{pmatrix},  
\end{equation}
which leads to the satisfaction of Condition \ref{cond:2-ssc-5}. Thus we have verified all the convergence criteria.
\end{proof}

\begin{remark}
Since Assumptions \ref{R1}-\ref{R3} are admitted by the compressible Navier-Stokes equations \cite{kawashima1984systems,kawashima1988}, one can obtain an effective relaxation model as \eqref{15} and then use the numerical schemes for hyperbolic systems to solve it efficiently. On the other hand, if we adopt a degenerate diffusion term as in \eqref{11} for systems \eqref{equ:examp-312-limit} and \eqref{equ:dk-1}, the validity of corresponding relaxation models may be justified as above. These are left for our future work.
\end{remark}

\section{Conclusions and remarks}
\label{sec5}

In this work, we systematically investigated the validity of several representative relaxation models that arise from numerical schemes for hyperbolic-parabolic systems. These models, formulated as hyperbolic systems with stiff source terms involving a small parameter, are ubiquitous in multiscale modeling and designing numerical schemes. 
By verifying the convergence criteria in \cite{yong1999singular,lattanzio2001hyperbolic,peng2025convergence} for general hyperbolic relaxation systems, we justified the validity of five representative relaxation models, including both direct relaxation models and kinetic approximation models.
Namely, with these criteria satisfied, solutions of the relaxation models converge to those of the target hyperbolic-parabolic systems in the relaxation limit. This provides a solid mathematical foundation for the consistency and convergence of the numerical schemes based on such relaxation models.
Moreover, we proposed a new relaxation model for general multi-dimensional hyperbolic-parabolic systems.  By some mild assumptions on the system, we show that the proposed relaxation model satisfies the convergence criteria and thus its approximation validity is guaranteed. Unlike the existing relaxation models for special hyperbolic-parabolic systems, our new relaxation model is valid for general systems. We believe that the new model is promising for broader applications in multiscale modeling and developing efficient numerical schemes for complex dissipative systems.



\end{document}